\documentclass[12pt, A4paper]{amsproc}
\usepackage{amsmath, amsthm, amssymb, tikz}
\usepackage{xcolor}
\usepackage{url, hyperref}
\usepackage{anysize}
\marginsize{1.8cm}{1.8cm}{2cm}{2.5cm}
\hypersetup{citecolor=blue, linkcolor=blue, colorlinks=true}
\usetikzlibrary{shapes,positioning,scopes}

 \newtheorem{theorem}{Theorem}[section]
 \newtheorem{lemma}[theorem]{Lemma}
 \newtheorem{corollary}[theorem]{Corollary}

\theoremstyle{remark}
\newtheorem{definition}[theorem]{Definition}

\DeclareMathOperator{\inc}{\mathbin{\mathrm{I}}}

\renewcommand\le{\leqslant}
\renewcommand\ge{\geqslant}

\title{Bruck nets and partial Sherk planes}
\author{John Bamberg}
\address{ %
\emph{All authors}:
Centre for the Mathematics of Symmetry and Computation,
School of Mathematics and Statistics,
The University of Western Australia,
Crawley, W.A. 6009, Australia}

\author{Joanna B. Fawcett}
\address[Fawcett]{%
\emph{Current address of J. B. Fawcett}:
Department of Pure Mathematics and Mathematical Statistics,
Centre for Mathematical Sciences,
 University of Cambridge,
 Wilberforce Road, Cambridge CB3 0WB, UK}

\author[Lansdown]{Jesse Lansdown}

\email{John.Bamberg@uwa.edu.au}
\email{j.b.fawcett@dpmms.cam.ac.uk}
\email{Jesse.Lansdown@research.uwa.edu.au}

\thanks{The first author acknowledges the support of the Australian Research Council (ARC) 
Future Fellowship FT120100036. 
The second author acknowledges the support of the ARC Discovery Grant DP130100106. 
The third author acknowledges the support of the ARC Discovery Grant DP0984540.}



\begin{document}
\maketitle

\begin{abstract}
In Bachmann's \emph{Aufbau der Geometrie aus dem Spiegelungsbegriff} (1959), it was 
shown that a finite metric plane is a Desarguesian affine plane of odd order equipped with a perpendicularity
relation on lines, and conversely. Sherk (1967) generalised this result to characterise the 
 finite affine planes of odd order by removing the `three reflections axioms' from a metric plane.
We show that one can obtain a larger class of natural finite geometries, the so-called \emph{Bruck nets} of
even degree,  by weakening Sherk's axioms to allow non-collinear points.
\end{abstract}

\section{Introduction}

Bruck nets generalise the notion of parallelism in affine planes by extending parallelism to partial linear spaces (a definition is given in Section \ref{prelim}), and they are of great interest due to their application to the study of mutually orthogonal latin squares. 
An affine plane with some parallel classes removed is an example of a Bruck net, but there exist Bruck nets that do not arise this way (see, for example, \cite{Bruen:1972aa}).
Sherk characterised affine planes of odd order as finite linear spaces with perpendicularity \cite{Sherk}, generalising the work of Bachmann who characterised finite Desarguesian affine planes of odd order as finite metric planes \cite{Bachmann}. In this paper we further generalise the work of Sherk and Bachmann by characterising Bruck nets of even degree as finite partial linear spaces with perpendicularity. Thus there is a sense in which the concepts of parallelism and perpendicularity are equivalent in finite partial linear spaces. Indeed, we prove that two lines are parallel precisely when they have a common perpendicular.

Let $\mathcal{G}$ be a geometry consisting of a set of points, a set of lines, a binary relation \emph{incidence} between points and lines, denoted by $\inc$, and a binary relation \emph{perpendicularity} between lines, denoted by $\perp$. Two lines are said to \emph{intersect} if they are each incident with a common point. We call $\mathcal{G}$ a \emph{Sherk plane} after Sherk \cite{Sherk} if it satisfies the following six axioms.

\begin{enumerate}
\item[{\bf A.}] Two distinct points are incident with a unique line.
\item[{\bf B1.}] For lines $\ell$ and $m$, if $\ell \perp m$ then $m \perp \ell$.
\item[{\bf B2.}] Perpendicular lines intersect in at least one point.
\item[{\bf B3.}] Given a point $P$ and line $\ell$, there exists at least one line $m$ such that $m \inc P$ and $m \perp \ell$.
\item[{\bf B4.}] Given a point $P$ and line $\ell$ such that $\ell \inc P$, there exists a unique line $m$ such that $m \inc P$ and $m \perp \ell$.
\item[{\bf B5.}] There exist lines $x,y$ and $z$ such that $x \perp y$, $x \not \perp z$, $y \not \perp z$ and $x, y$ and $z$ do not intersect at a common point.
\end{enumerate}

These axioms are equivalent to the incidence and perpendicularity axioms of a metric plane \cite{Bachmann}. Thus a Sherk plane generalises the notion of a metric plane by removing the powerful three reflections axioms. Any geometry (with a non-empty point set) satisfying Axiom A is a \emph{linear space}. This concept can be generalised to that of a \emph{partial linear space} by weakening Axiom A to the following.

\begin{enumerate}
\item[{\bf A*.}] Two distinct points are incident with at most one line.
\end{enumerate}

We can similarly generalise Sherk planes from a linear space equipped with perpendicularity to a partial linear space equipped with perpendicularity. We call a geometry $\mathcal{G}$ satisfying Axiom A* and Axioms B1 - B5 a \emph{partial Sherk plane}. Such a geometry is \emph{finite} if it possesses only finitely many points. 
A point of $\mathcal{G}$ is \emph{thick} if it is incident with at least three lines, and \emph{thin} otherwise.

We now state the main result of this paper.

\begin{theorem} \label{main}
The following are equivalent.
\begin{enumerate}
	\item A finite partial Sherk plane in which
	some line is incident only with thick points.
	\item A finite Bruck net of degree $r$ where $r$ is even and $r>2$.
\end{enumerate}
\end{theorem}

In fact, we prove in Lemmas \ref{NisPSP} and \ref{Tau} that perpendicularity can only have the form of Definition~\ref{PerpInNet}. Further, the thickness condition cannot be removed from Theorem \ref{main}(1), as there are infinitely many examples of finite partial Sherk planes in which every line is incident with a thin point.

The format of this paper is as follows. In Section \ref{prelim}, we define Bruck nets and give some of their fundamental properties. In Section \ref{ProofSection}, we prove Theorem \ref{main}, and in Section \ref{further}, we give a method for constructing partial Sherk planes that do not satisfy the thickness condition of Theorem \ref{main}(1).

\section{Bruck nets} \label{prelim}

A \emph{(Bruck) net} is a partial linear space satisfying the following three axioms.
\begin{enumerate}
\item[{\bf N1:}] For a line $\ell$ and point $P$ not incident with $\ell$, there exists a unique line $m$ such that $m \inc P$ and $m$ and $\ell$ do not intersect.
\item[{\bf N2:}] For every line, there exist two points not incident with it.
\item[{\bf N3:}] For every point, there exist two lines not incident with it.
\end{enumerate}
The line $m$ in Axiom N1 is  the \emph{parallel to $\ell$ at the point $P$}. Two lines are  \emph{parallel} to each other if they do not intersect. Parallelism is an equivalence relation, and the equivalence classes of parallel lines are called \emph{parallel classes}  \cite[p.141]{Dembowski}. We denote the parallel class of $\ell$ by $[ \ell]$. A finite net has the following properties  \cite{BruckII}, \cite[p.141]{Dembowski}.

\begin{lemma} \label{NetProperties}
In a finite Bruck net, there exist integers $r$ and $n$ such that the following are true.
\begin{enumerate}
\item There are $r$ lines incident with every point. \label{rLines}
\item There are $r$ parallel classes. \label{pClasses}
\item There are $n$ points on every line.
\item There are $n$ lines in each parallel class. \label{nLinesParallelClass}
\item There are $n^2$ points in total. \label{netPoints}
\item There are $rn$ lines in total.
\item Two lines from distinct parallel classes intersect in a unique point. \label{intersect}
\item Every point is incident with exactly one line of each parallel class. \label{IncOnePClass}
\end{enumerate}
\end{lemma}

We call a finite Bruck net with $r$ lines through any point and $n$ points on any line an \emph{$(n, r)$-net} and say it has \emph{order $n$} and \emph{degree $r$}. We extend Sherk's definition \cite{Sherk} of perpendicularity for finite affine planes of odd order  as follows.

\begin{definition} \label{PerpInNet}
Let $\mathcal{N}$ be an $(n,r)$-net where $r$ is even and $r >2$. Since $\mathcal{N}$ has an even number of parallel classes by Lemma \ref{NetProperties}(\ref{pClasses}), we may choose an involution $\tau$ that acts fixed-point-freely on the set of parallel classes. For lines $\ell$ and $m$, define $\ell \perp m$ if and only if $[\ell]^\tau = [m]$.
\end{definition}

\begin{lemma} \label{NisPSP}
Let $\mathcal{N}$ be an $(n,r)$-net with $r$ even,    $r>2$ and perpendicularity  as in  Definition $\ref{PerpInNet}$. Then $\mathcal{N}$ is a partial Sherk plane in which every point is incident with $r$ lines.
\end{lemma}

\begin{proof}
Axiom A* holds by definition. Axiom B1 holds since $\tau$ is an involution. By Lemma \ref{NetProperties}(\ref{intersect}), any two lines in different parallel classes meet in a point, so Axiom B2 holds. Let $\ell$ be a line and $P$ a point. By Lemma \ref{NetProperties}(\ref{IncOnePClass}), there exists $m$ in $[\ell]^\tau$ such that $m \inc P$. Now $[\ell]^\tau = [m]$, so $\ell \perp m$ and Axiom B3 holds. Moreover, any such $m$ is unique by Lemma \ref{NetProperties}(\ref{IncOnePClass}), so Axiom B4 holds. There are at least four parallel classes by Lemma $\ref{NetProperties}(\ref{pClasses})$. Let $\ell$ be a line and choose $m$ in $[\ell]^\tau$. Now $\ell$ and $m$ are perpendicular by definition, and they intersect in a unique point $P$ by Lemma $\ref{NetProperties}(\ref{intersect})$. Furthermore, $n >1$ by Lemma $\ref{NetProperties}(\ref{netPoints})$ and Axiom N2, so by Lemma $\ref{NetProperties}(\ref{nLinesParallelClass})$ and $(\ref{IncOnePClass})$, there exists a line $g$ not in $[\ell]\cup[m]$  such that $g$ is not incident with $P$. Thus B5 holds. By Lemma \ref{NetProperties}(\ref{rLines}), there are exactly $r$ lines incident with any point.
\end{proof}

\section{Proof of Theorem 1.1} \label{ProofSection}

If Theorem \ref{main}(2) holds, then Theorem \ref{main}(1) holds by Lemma \ref{NisPSP}.
We assume for the remainder of the section that $\mathcal{G}$ is a partial Sherk plane.
Let $v$ be the total number of points, and $b$ the total number of lines. Since $v>0$ and $b>0$ by Axioms B5 and B2, 
it follows from Axiom B3 that there is at least one line on every point, and from Axioms B3 and B2 that every line is incident with some point. We denote the number of lines through a point $P$ by $r_P$, and the number of points on a line $\ell$ by $n_\ell$. Note that Lemmas \ref{NoSelfPerp}, \ref{rEven},   \ref{Sherk6} and \ref{nPointsPerLine}   are similar to results in \cite{Sherk} but their proofs are included for completeness; in addition, the proofs of Lemmas \ref{nPointsPerLineNotPerp} and \ref{rn} contain arguments from  \cite{Sherk}. In Section \ref{s:basics}, we prove some basic results that do not require the assumption that some line is incident only with thick points,  and in Section \ref{s:}, we assume this thickness condition and prove that Theorem \ref{main}(2) holds.

\subsection{The basics.}
\label{s:basics}

\begin{lemma} \label{NonCollinear}
There exist three non-collinear points, and some line has at least three points.
\end{lemma}

\begin{proof}
Let $x$, $y$ and $z$ be lines as in Axiom B5. Then $x$ and $y$ intersect at a point $P$ by Axiom B2. By Axiom B3, there exists a  perpendicular\footnote{For a point $P$ and line $\ell$, a  \emph{perpendicular from $P$ to $\ell$} is a line incident with  $P$ and perpendicular to  $\ell$.} $g$ from $P$ to $z$. Clearly $g$ is not $x$ or $y$ as $z$ is not perpendicular to either of these lines. Let $Q$ be the point of intersection of $z$ and $g$. Now $Q \neq P$ since $z$ is not incident with $P$ by Axiom B5. Let $h$ be a perpendicular from $Q$ to $y$. Note that $h \neq z$ as $z$ is not perpendicular to $y$, nor is $h$ equal to $x$ as this would imply that $Q \inc x$ and therefore that $g = x$ by Axiom A*, a contradiction. Furthermore $h \neq g$ or else $P \inc h$, in which case $h=x$ by Axiom B4. Denote the intersection of $h$ and $y$ by $R$. Now $R$ and $P$ are distinct since $g$ and $h$ are distinct, and $R$ and $Q$ are distinct since $g$ and $y$ are distinct. Moreover, $P$, $Q$ and $R$ are non-collinear by Axiom A* since $g\neq y$. Thus there exist three non-collinear points. Let $k$ be a  perpendicular from $R$ to $g$, and let $S$ be the intersection of $k$ and $g$. Now $S \neq P$ or else $k=y$ by Axiom A*, which would imply that $g = x$ by Axiom B4, a contradiction. Furthermore, $S \neq Q$ or else $k=z$ since $g$ is perpendicular to $z$ at $Q$ and perpendicular to $k$ at $S$, but then $R$ and $Q$ are on $z$, so $z= h$, a contradiction. Thus $P$, $Q$ and $S$ are pairwise distinct points on the line $g$.
\end{proof}

\begin{lemma} \label{NoSelfPerp}
No line is perpendicular to itself.
\end{lemma}

\begin{proof}
Let $\ell$ be a line. By Lemma \ref{NonCollinear}, there exists a point $P$ not on $\ell$. By Axiom B3, there exists a perpendicular $g$ from $P$ to  $\ell$. By Axiom B2, there exists a point $Q$ at the intersection of $\ell$ and $g$. Now $g$ is the unique line incident with $Q$ and perpendicular to $\ell$ by Axiom B4. Clearly $g \neq \ell$, as $g$ is incident with $P$ and $\ell$ is not. Thus $\ell$ cannot be perpendicular to $\ell$.
\end{proof}

\begin{lemma} \label{rEven}
The number of lines on any point is even.
\end{lemma}

\begin{proof}
Let $P$ be a point. By Axiom B4, every line incident with $P$ has a unique perpendicular at $P$. It then follows from Lemma \ref{NoSelfPerp} that $r_P$ is even.
\end{proof}

\begin{lemma} \label{NoPerps}
Let $\ell$ and $m$ be lines that are not perpendicular. Let $P$ and $Q$ be distinct points that are incident with $\ell$, and let $g$ and $h$ be perpendiculars from $P$ and $Q$ respectively to $m$. Then $g\neq h$, and $g$ and $h$ intersect $m$ in distinct points. 
\end{lemma}

\begin{proof}
If $g = h$, then both $P$ and $Q$ are incident with $g$. Thus $g = \ell$ by Axiom A*, so $\ell \perp m$, a contradiction. Hence $g$ and $h$ are distinct. It then follows from Axioms B2 and  B4 that $g$ and $h$ meet $m$ in distinct points.
\end{proof}

\begin{lemma} \label{NoPerps2}
Given any line $\ell$, there are exactly $n_\ell$ lines that are perpendicular to $\ell$.
\end{lemma}

\begin{proof}
 By Axiom B4, each of the $n_\ell$ points of $\ell$ has a unique perpendicular to $\ell$, and these perpendiculars are pairwise distinct by Axiom A* and Lemma \ref{NoSelfPerp}. By Axiom B2, no other lines are perpendicular to $\ell$. Thus there are exactly $n_\ell$ lines that are perpendicular to $\ell$.
\end{proof}

\begin{lemma} \label{nPointsPerLineNotPerp}
Let $\ell$ and $m$ be lines that are not perpendicular. Then $n_\ell=n_m$.
\end{lemma}

\begin{proof}
Assume without loss of generality that $n_\ell\ge n_m$. There are at least $n_\ell$ lines intersecting $\ell$ that are perpendicular to $m$ by Axiom B3 and Lemma \ref{NoPerps}, and there are exactly $n_m$ lines that are perpendicular to $m$ by Lemma \ref{NoPerps2}, so $n_\ell=n_m$.
\end{proof}

\begin{lemma} \label{Sherk6}
Let $\ell$ and $m$ be perpendicular lines. Then $\ell$ intersects all lines of $\mathcal{G}$ except possibly the lines that are perpendicular to $m$.
\end{lemma}

\begin{proof}
Let $g$ be a line that is not perpendicular to $m$. Then $n_g=n_m$ by Lemma \ref{nPointsPerLineNotPerp}. By Axiom B3 and Lemmas \ref{NoPerps} and \ref{NoPerps2}, there are exactly $n_g$ perpendiculars from the $n_g$ points of $g$ to $m$, and every perpendicular to $m$ arises in this way. Since $\ell$ is perpendicular to  $m$, the line $\ell$ intersects $g$.
\end{proof}

\begin{lemma} \label{PoleAllOrOne}
For a point $P$ and line $\ell$ such that $P$ is not incident with $\ell$, either every line incident with $P$ is perpendicular to $\ell$, or there is a unique perpendicular from  $P$ to $\ell$.
\end{lemma}

\begin{proof}
Suppose that some line $m$ is incident with $P$ but not perpendicular to $\ell$. By Lemma \ref{nPointsPerLineNotPerp}, $n_\ell=n_m$. Write $n=n_\ell$. Let $P_1, \ldots ,P_n$ be the points on $m$, and let $h_i$ be a perpendicular from $P_i$ to $\ell$ for $1\le i\le n$.
Then for all $i\ne j$, the lines $h_i$ and $h_j$ are distinct and meet $\ell$ in distinct points by Lemma \ref{NoPerps}.
 Say $P=P_i$. If
$h_i$ is the unique perpendicular from $P$ to $\ell$, then we are done, so suppose that $g$ is another perpendicular from $P$ to $\ell$. Since $\ell$ has $n$ points, Axiom B4 implies that $g=h_j$ for some $j \neq i$. Then $P_i$ and $P_j$ are both incident with $g$ and $m$, so $g=m$ by Axiom A*, in which case $m$ is perpendicular to $\ell$, a contradiction.
\end{proof}

In the case where a point $P$ is not incident with a line $\ell$ and there is more than one perpendicular from $P$ to $\ell$, we say that \emph{$P$ is a pole of $\ell$} and \emph{$\ell$ is a polar of $P$}. 

\begin{lemma}\label{polesperline}
Let $\ell$ and $h$ be perpendicular lines intersecting at a point $P$. Then the
number of poles of $\ell$ on $h$ equals the number of polars of $P$.  
\end{lemma}

\begin{proof}
 By Lemma \ref{PoleAllOrOne}, every  polar $g$ of $P$ is perpendicular to $\ell$ and $h$, so $g$ intersects $h$ in a point $Q_g$, and $Q_g$ is a pole of $\ell$. Thus we have a map from the set of polars of $P$ to the set of poles of $\ell$ on $h$ defined by $g\mapsto Q_g$, and this map is injective by Axiom B4. Now for a pole $Q$ of $\ell$ on $h$,  the unique perpendicular to $h$ at $Q$ is a polar of $P$ by Lemma \ref{PoleAllOrOne}, so this map is a bijection. 
\end{proof}

\begin{lemma}\label{collinear}
Let $P$ and $Q$ be two distinct points such that $P$ is not a pole for any line through $Q$, and $Q$ is not a pole for any 
line through $P$. Then $r_P=r_Q$.
\end{lemma}

\begin{proof}
Let $m$ be a line through $P$. Since $Q$ is not a pole of $m$, there is a unique perpendicular $\ell_m$ from $Q$ to $m$. 
Thus we have a map from the set of lines on $P$ to the set of lines on $Q$ defined by $m\mapsto \ell_m$. If 
$\ell_{m_1}=\ell_{m_2}$ for lines $m_1$ and $m_2$ on $P$, then $\ell_{m_1}$ is perpendicular to $m_1$ and $m_2$, so either 
$m_1=m_2$, or $P$ is a pole of $\ell_{m_1}$. This latter possibility contradicts our assumption, so the map is injective 
and $r_P\le r_Q$. By reversing the roles of $P$ and $Q$, we also have that $r_Q\le r_P$.
\end{proof}

\begin{lemma}\label{constant_on_line_perp}
Let $P$ and $Q$ be distinct points on a line $\ell$, and let $m_P$ and $m_Q$ be the perpendiculars to $\ell$ at $P$ and $Q$ respectively. If $m_P$ and $m_Q$ are not perpendicular, then $r_P=r_Q$.
\end{lemma}

\begin{proof}
Suppose that $Q$ is a pole of a line $m$ through $P$. Now every line through $Q$ is perpendicular to $m$ by Lemma \ref{PoleAllOrOne}; 
in particular, $\ell$ is perpendicular to $m$. Since $m_P$ and $m$ are both perpendicular to $\ell$ at $P$, it follows that $m_P=m$. Now $m_Q$ is a line through $Q$, so $m_Q$ is perpendicular to $m=m_P$ by Lemma \ref{PoleAllOrOne}, a contradiction. Thus $Q$ is not the pole of any line through $P$, and by symmetry, 
$P$ is not the pole of any line through $Q$. By Lemma \ref{collinear}, $r_P=r_Q$. 
\end{proof}

\subsection{Assuming the thickness condition.}
\label{s:}
We now assume that some line is incident only with thick points.

\begin{lemma}
\label{thickpoints}
Every point is thick.
\end{lemma}

\begin{proof}
Let $\ell$ be the line that is incident only with thick points. Suppose for a contradiction that there exists a point $P$ such that $r_P = 2$. Let $g$ and $h$ be the two lines incident with $P$. By Axiom B3, every line other than $g$ and $h$ must be perpendicular to at least one of $g$ or $h$. Without loss of generality,  suppose that $\ell$ is perpendicular to $g$, and let $R$ be the intersection point of  $\ell$ and $g$. Since $r_R\ge 3$, there is some 
other line $x$ on $R$,  and by Axiom B4, $x$ is not perpendicular 
to $g$, so $x$ must be perpendicular to $h$; let $S$ be the intersection point of $x$ and $h$. Since $x$ and $g$ are both perpendicular to $h$, $R$ is a pole of $h$. By Lemma \ref{PoleAllOrOne}, $\ell$ must be 
perpendicular to $h$ at a point $Q$. Since $g$ is the perpendicular to $h$ at $P$ and $x$ is the perpendicular to $h$ at $S$, Lemma \ref{constant_on_line_perp} implies that $r_P =r_S= r_Q$, but $r_Q\ge 3$, a contradiction.
\end{proof}

\begin{lemma} \label{nPointsPerLine}
The number of points incident with any line is constant.
\end{lemma}

\begin{proof}
Let $\ell$ and $m$ be lines. If $\ell$ is not perpendicular to $m$, then $n_\ell=n_m$ by Lemma \ref{nPointsPerLineNotPerp}.  Otherwise, $\ell$ is perpendicular to $m$, in which case $\ell$ and $m$ intersect at a point $P$. 
Since $r_P\ge 3$ by Lemma \ref{thickpoints}, there exists a line $g$ that is  incident with $P$ and distinct from $\ell$ and $m$. Now $g$ is not perpendicular to $\ell$ or $m$ by Axiom  B4, so $n_\ell=n_g=n_m$ by Lemma \ref{nPointsPerLineNotPerp}. 
\end{proof}

Let $n$ denote the number of points on any line.
We have the following consequence of Lemmas \ref{NonCollinear} and  \ref{nPointsPerLine}.

\begin{corollary} \label{nGeq3}
$n \ge 3$.
\end{corollary}

\begin{lemma} \label{N2N3}
Axioms $N2$ and $N3$ hold.
\end{lemma}

\begin{proof}
Let $\ell$ be a line. There exists a point $P$ on $\ell$, and there is another line $g$ incident with $P$. By Corollary \ref{nGeq3}, there are at least two additional points on $g$, and by Axiom A* they are not on $\ell$. Thus Axiom N2 holds. Let $P'$ be a point. There exists a line $\ell'$ incident with $P'$. By Corollary \ref{nGeq3}, there exists a point $Q$ on $\ell'$ not equal to $P'$, and since $r_{Q}\ge 3$ by Lemma \ref{thickpoints}, there are two lines other than $\ell'$ incident with $Q$, which by Axiom A* are not incident with $P'$. Thus Axiom N3 holds.
\end{proof}

\begin{lemma}\label{constant_r}
The number of lines through  any point is constant.
\end{lemma}

\begin{proof}
 Let $P$ and $Q$ be distinct points on a line $\ell$. Let $m_P$ and $m_Q$ be the perpendiculars to $\ell$ at $P$ and $Q$ respectively. If $m_P$ and $m_Q$ are not perpendicular, then $r_P=r_Q$ by  Lemma \ref{constant_on_line_perp},  so we assume otherwise. Let $R$ be the point of intersection of $m_P$ and $m_Q$. Since $r_R\ge 3$ by Lemma \ref{thickpoints} and $R$ is a pole of the line $\ell$, Lemma \ref{PoleAllOrOne} implies that there exists a point $S$ on $\ell$ and a line $m_S$ on $R$ and $S$ such that $m_S$ is perpendicular to $\ell$, and 
$m_S$ is not perpendicular to $m_P$ or $m_Q$. By Lemma \ref{constant_on_line_perp}, $r_P=r_S=r_Q$. 

Now suppose that $P$ and $Q$ are non-collinear points.  There exists a line $m$ through $Q$, and by Axiom B3, there is a point $R$ and line $\ell$ such that $P$ and $R$ are incident with $\ell$ and $R$ is incident with $m$. Then $r_P=r_R=r_Q$ by the above argument.
\end{proof}

Let   $r$ denote the number of lines through any point.

\begin{lemma} \label{Npoles}
Every line has exactly $(n^2-v)/(r-1)$ poles.
\end{lemma}

\begin{proof}
Let $\ell$ be a line, and let $N$ be the number of poles of $\ell$. Let $t$ be the number of pairs $(P, m)$ such that $P$ is a point not on $\ell$, and $m$ is a line incident with $P$ and perpendicular to $\ell$. Recall that $v$ is the total number of points in $\mathcal{G}$. There are $N$ poles of $\ell$, each   incident with $r$ lines, and all such lines are perpendicular to $\ell$ by Lemma \ref{PoleAllOrOne}. Moreover, there are $v -n -N$ points not incident with $\ell$ that are not poles of $\ell$, each with exactly one perpendicular to $\ell$ by Lemma \ref{PoleAllOrOne}. Thus $t =Nr+ (v -n -N)$. On the other hand, there are $n$ lines perpendicular to $\ell$, each with $n-1$ points not on $\ell$, so $t = n(n-1)$. Hence $N=(n^2-v)/(r-1)$.
\end{proof}

We shall denote the number of poles of a line by $N$.

\begin{lemma}\label{numberlines}
Every point has exactly $(nr - b)/(r - 1)$ polars.
\end{lemma}

\begin{proof}
Let $P$ be a point, and let $M$ be the number of polars of $P$. Let $t$ be the number of pairs $(Q, \ell)$ such that $\ell$ is a line not incident with $P$, and $Q$ is a point incident with $\ell$ such that the perpendicular from $Q$ to $\ell$ is incident with $P$. Recall that $b$ is the total number of lines in $\mathcal{G}$. 
There are $M$ polars of $P$, and each such polar $\ell$ is incident with $r$ points $Q$ such that the perpendicular from $Q$ to $\ell$ is incident with $P$ by Lemma \ref{PoleAllOrOne}.
Moreover, there are $b-r-M$ lines not incident with $P$ that are not  polars of $P$, and for each such line $\ell$, there is a unique point $Q$ incident with $\ell$ such that the perpendicular from $Q$ to $\ell$ is incident with $P$ by Lemma \ref{PoleAllOrOne}. Thus $t= Mr+ (b-r-M)$. On the other hand, there are $r$ lines incident with $P$, each of which is incident with $n-1$ points besides $P$, and each such line and point determines a unique perpendicular line not incident with $P$, so $t=r(n-1)$. Hence $M=(nr-b)/(r-1)$.
\end{proof}

We shall denote the number of polars of a point by $M$.

\begin{lemma}\label{flags}
$Nr=Mn$.
\end{lemma}

\begin{proof}
Every point is incident with $r$ lines and every line is incident with $n$ points, so $vr=bn$. Similarly, every line has $N$ poles and every point has $M$ polars, so $bN=vM$. Thus $Nr=Mn$.
\end{proof}

\begin{lemma}\label{rn}
Suppose that a line $\ell$ has a pole $P$. Then $r<n$.
\end{lemma}

\begin{proof}
Suppose (by way of contradiction) that $r \ge n$. By Lemma \ref{PoleAllOrOne}, the $r$ lines incident with $P$ are all perpendicular to $\ell$, so 
$r=n$. By Lemma \ref{NoPerps2}, every line that is perpendicular to $\ell$ is incident with $P$. By Axiom B3, every point of $\mathcal{G}$ lies on a perpendicular to $\ell$. Each of the $n$ perpendiculars to $\ell$ is incident with $n-1$ points other than $P$. Thus 
there are $n(n-1)$ points collinear with $P$, and therefore any point, and we conclude that $\mathcal{G}$ is a linear space.
Moreover, $\mathcal{G}$ has $v=n(n-1)+1$ points, so $\mathcal{G}$ is what is known as a `$2-(n(n-1)+1,n,1)$ design' 
and hence a projective plane (c.f., \cite[p. 138]{Dembowski}).
Further, by Lemmas \ref{Npoles} and  \ref{flags}, $N=M=1$.
Therefore, every point $P$ has a unique polar and every line has a unique pole, so $\mathcal{G}$ is equipped with a polarity\footnote{A \emph{polarity} is  an involutary map
from points to lines that preserves incidence.
A point $Q$ is \emph{absolute} with respect
to a polarity $\rho$ if $Q$ is incident with $Q^\rho$.} that maps each point to its polar, and this polarity has no absolute points. However, this contradicts
Baer's Theorem on polarities
of finite projective planes \cite[p. 82]{Baer:1946aa} since our projective plane $\mathcal{G}$ has at least three points on every line.
Hence $r<n$.
\end{proof}

\begin{lemma}\label{lowerbound}
Suppose that a line $\ell$ has a pole $P$. Then $n-r+1\le N$.
\end{lemma}

\begin{proof}
By Lemma \ref{rn}, we have $r<n$.
Then there are pairwise distinct lines $h_1,\ldots, h_m$ where $m=n-r$ such that $h_i$ is not incident with $P$
and $h_i$ is perpendicular to $\ell$. For $1\le i\le m$, there exists a perpendicular $a_i$ from $P$ to 
$h_i$. Let $Q_i$ be the intersection of $a_i$ and $h_i$. Then $P,Q_1,\ldots,Q_m$ are pairwise distinct points, none of which are incident with $\ell$ (by Axiom B4). Moreover, $a_i$ and $h_i$ are distinct lines that are incident with $Q_i$ and perpendicular to $\ell$ for $1\le i\le m$, so $Q_1,\ldots,Q_m$ are poles of $\ell$. Hence $n-r+1\le N$. 
\end{proof}

\begin{theorem}\label{NoPoles}
No line has a pole.
\end{theorem}

\begin{proof}
Suppose for a contradiction that some line $\ell$ has a pole $P$.
Then $r<n$ by Lemma \ref{rn}. Let $W$ be the number of poles
of $\ell$ that are not equal to or collinear with $P$. Each of the $r$ lines on $P$ is perpendicular to $\ell$ by Lemma \ref{PoleAllOrOne}, so 
by Lemma \ref{polesperline},    
\[
W=N-1-(M-1)r.
\]
Let $t$ be the number of pairs $(Q,h)$ where $Q$ is a pole of $\ell$ not equal to or collinear with $P$, and $h$ is a line 
perpendicular to $\ell$ that is incident with $Q$ but not $P$. Each of the $n-r$ points 
of $\ell$ not collinear with $P$ has a unique perpendicular to $\ell$, and each
such perpendicular is incident with $M$ poles of $\ell$ by Lemma \ref{polesperline}, so $t=M(n-r)$. 
By another count for $t$,
\[
Wr=M(n-r).
\]
The two equations above yield $(N-1-(M-1)r)r=M(n-r)$.
Now $Nr=Mn$ by Lemma \ref{flags}, so $Mr(r-1)=r(r-1)$. Since $r>1$, it follows that $M=1$ and $N=n/r$.
By Lemma \ref{lowerbound}, $n-r+1\le N=n/r$. Thus $(n-r)(r-1)\le 0$, so $n\le r$, a contradiction.
\end{proof}

The following extension of Axiom B3 is a consequence of Theorem \ref{NoPoles}.

\begin{corollary} \label{UniquePerp}
For a line $\ell$ and point $P$, there exists a unique line that is perpendicular to $\ell$ and incident with $P$.
\end{corollary}

\begin{lemma} \label{PerpIffParallel}
Two lines are parallel if and only if they have a common perpendicular.
\end{lemma}

\begin{proof}
If two lines do not have a common perpendicular, then by Lemma \ref{Sherk6} they must intersect. 
Conversely, if two lines do have a common perpendicular, then they cannot intersect by Theorem \ref{NoPoles}.
\end{proof}

\begin{lemma} \label{ExistenceOFParallel}
Axiom $N1$ holds.
\end{lemma}

\begin{proof}
Let $\ell$ be a line and $P$ a point not incident with $\ell$. By Corollary \ref{UniquePerp}, there exists a unique  perpendicular $g$ from $P$ to $\ell$. By Axiom B4, there exists a unique line $m$ incident with $P$ and perpendicular to $g$. Clearly $m \neq \ell$ as $m$ is incident with $P$ but $\ell$ is not, so $m$ is a line that is incident with $P$ and parallel to $\ell$ by Lemma \ref{PerpIffParallel}. Say $m'$ is another line that is incident with $P$ and parallel to $\ell$. Since $g$ and $\ell$ are perpendicular, so are $m'$ and $g$ by Lemma \ref{Sherk6}. But now $m$ and $m'$ are both perpendicular to $\ell$ at $P$, so $m = m'$ by Axiom B4.
\end{proof}

The result of Lemmas \ref{rEven}, \ref{N2N3} and \ref{ExistenceOFParallel} is the following.

\begin{corollary}
$\mathcal{G}$ is a finite Bruck net of degree $r$ where $r$ is even and $r>2$.
\end{corollary}

Now that Theorem \ref{main}(2) holds, we prove that perpendicularity has the form of Definition \ref{PerpInNet}.

\begin{lemma} \label{Tau}
There exists a fixed-point-free involution $\tau$ on the set of parallel classes such that $\ell \perp m$ if and only if $[\ell]^\tau = [m]$.
\end{lemma}

\begin{proof}
Let $\ell$ be a line. Denote the set of perpendiculars to $\ell$ by $X_\ell$. Now $X_\ell$ has size $n$, and any two lines in $X_\ell$ are parallel by Lemma \ref{PerpIffParallel}, so $X_\ell$ is a parallel class by Lemma \ref{NetProperties}(\ref{nLinesParallelClass}). Define a map $\tau$ on the set of parallel classes of $\mathcal{G}$ by mapping $[\ell]$ to $X_\ell$ for all lines $\ell$. If $[\ell] = [\ell']$, then $\ell$ and $\ell'$ have a common perpendicular, say $m$, by Lemma \ref{PerpIffParallel}, so $X_\ell = [m] = X_{\ell'}$. Hence $\tau$ is well defined. If $\ell$ and $\ell'$ are lines such that $X_\ell = X_{\ell'}$, then $\ell$ and $\ell'$ have a common perpendicular, so $[\ell] = [\ell']$ by Lemma \ref{PerpIffParallel}. Thus $\tau$ is a bijection. Moreover, $\tau$ has no fixed points since no line is perpendicular to itself by Lemma \ref{NoSelfPerp}, and $\tau$ is clearly an involution. Lastly, two lines $\ell$ and $m$ are perpendicular if and only if $m \in X_\ell$. This occurs precisely when $[m] = [\ell]^\tau$, as desired.
\end{proof}

Finally, we note that Sherk's theorem follows from Theorem \ref{main} (and Lemma \ref{NonCollinear})  upon assuming that $\mathcal{G}$ is a linear space.

\begin{corollary} \label{mainSherk}
The following are equivalent.
\begin{enumerate}
	\item A finite Sherk plane.
	\item A finite affine plane of odd order.
\end{enumerate}
\end{corollary}

\section{Partial Sherk planes that do not satisfy Theorem \ref{main}(1)}\label{further}

In this section, we give two methods for constructing   partial Sherk planes in which every line is incident with a thin point. This is done by adding points and lines to a given partial Sherk plane.

Let $\mathcal{G}$ be any partial Sherk plane, and let $k$ be a positive integer. Let $\mathcal{G}_k$ denote the geometry defined as follows. For each line $\ell$ of $\mathcal{G}$ and integer $i\in\{1,\ldots,k\}$, we add a new point $P_{\ell,i}$ to $\mathcal{G}$, and we extend the incidence relation of $\mathcal{G}$ so that $P_{\ell,i}$ is incident with $\ell$ but no other lines of $\mathcal{G}$. Now add  lines $g_1,\ldots,g_k$ such that, for each $i\in\{1,\ldots,k\}$, $g_i$ is incident with $P_{\ell,i}$ for all lines $\ell$ of $\mathcal{G}$, and no other points.  Lastly, extend the perpendicularity relation of $\mathcal{G}$ so that, for each $i\in \{1,\ldots,k\}$, $g_i$ is perpendicular to  every line of $\mathcal{G}$, every line of $\mathcal{G}$ is perpendicular to $g_i$, and $g_i$ is not perpendicular to $g_j$ for $j\in\{1,\ldots,k\}$.

Next, suppose that $k$ is even, and let $\mathcal{G}_k^*$ denote the geometry defined as follows. Add a new point $Q$ to $\mathcal{G}_k$, and extend the incidence relation of $\mathcal{G}_k$ so that $Q$ is incident with $g_1,\ldots,g_k$ and no other lines of $\mathcal{G}_k$. Lastly, change the perpendicularity relation of $\mathcal{G}_k$ so that $g_i$ is perpendicular to $g_{k-i+1}$ for $1\le i\le k$. See Figure \ref{fig:AG23} for  the case where $k=4$ and $\mathcal{G}$ is the affine plane of order $3$.

\begin{lemma}
 Let $\mathcal{G}$ be a (finite) partial Sherk plane, and let $k$ be a positive integer.  Then $\mathcal{G}_k$ is a (finite) partial Sherk plane. Further, if $k$ is even, then $\mathcal{G}_k^*$ is a (finite) partial Sherk plane.
 \end{lemma}
 
 \begin{proof}
 By construction, $\mathcal{G}_k$ and $\mathcal{G}_k^*$ ($k$ even) satisfy Axioms A$^*$, B1 and B2, and since Axiom B5 holds for $\mathcal{G}$, it also holds for $\mathcal{G}_k$ and $\mathcal{G}_k^*$.  It remains to show that Axioms B3 and B4 hold.
 
 Let $P$ and $\ell$ be a point and line of $\mathcal{G}_k$ or $\mathcal{G}_k^*$ respectively. We claim that there exists a line $m$ such that $m\inc P$ and  $m\perp\ell$, and that $m$ is the unique such line when $P\inc \ell$. This claim holds in $\mathcal{G}_k^*$ if $P=Q$, so we assume otherwise.   If $\ell\neq g_i$ for $1\le i\le n$, then since $\mathcal{G}$ satisfies Axioms B3 and B4, we may assume that  $P=P_{h,j}$ for some line $h$ of $\mathcal{G}$ and integer $j$, in which case the claim holds with $m=g_j$. Thus we may assume that $\ell=g_i$ for some $i$.  If $P=P_{h,j}$ for some line $h$ of $\mathcal{G}$ and integer $j$, then the claim holds with $m=h$. Otherwise, $P$ must be a point of $\mathcal{G}$, in which case we may take $m$ to be any line on $P$ (and such a line always exists). Thus Axioms B3 and B4 hold in $\mathcal{G}_k$ and $\mathcal{G}_k^*$.  
 \end{proof}
 
 \begin{figure}[ht]
\begin{center}
\includegraphics[page=1,height=7cm]{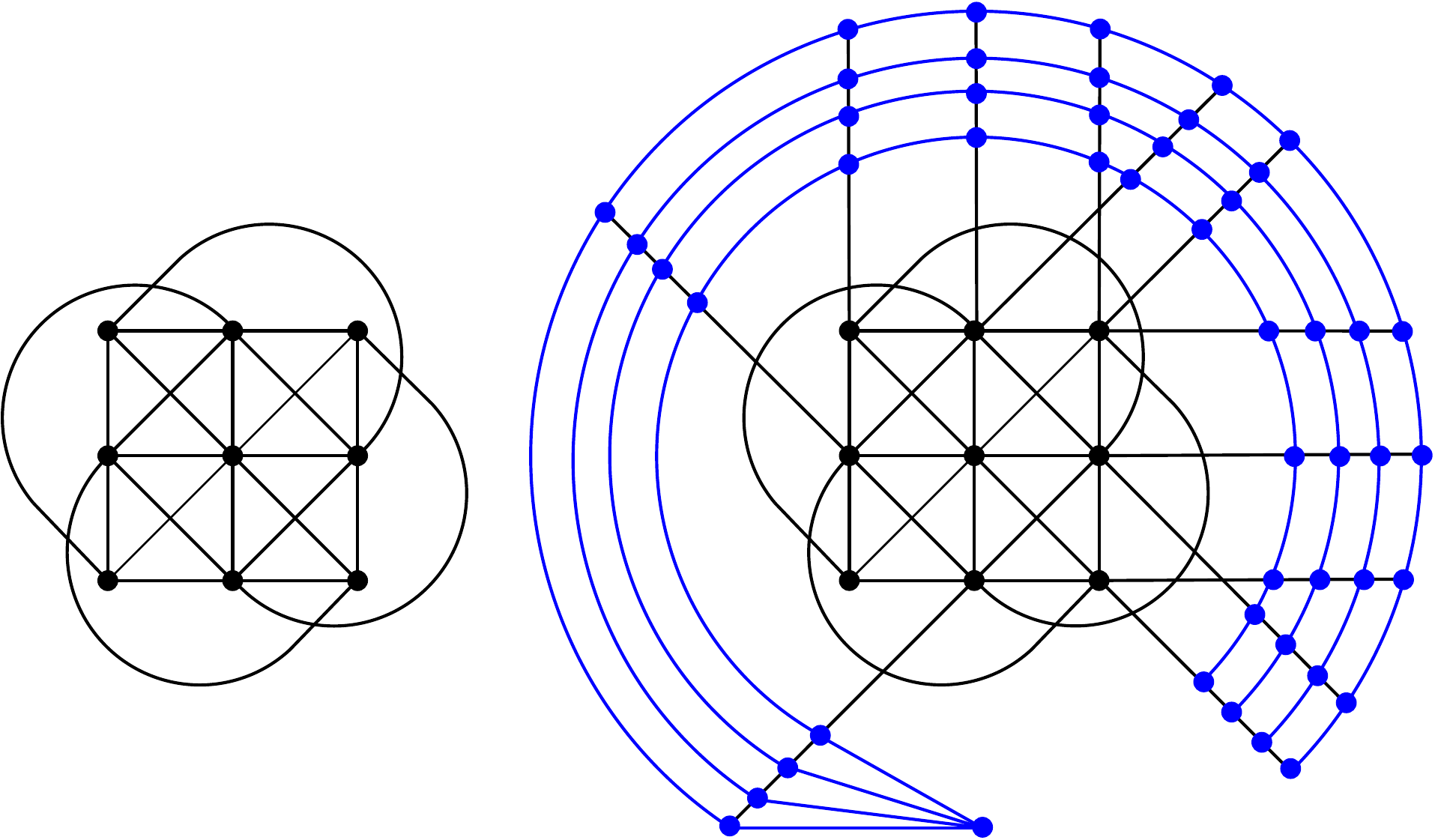}
\end{center}
\caption{The partial Sherk plane $\mathcal{G}_4^*$ (right) where $\mathcal{G}$ is the affine plane  of order $3$ (left).}
\label{fig:AG23}
\end{figure}

Every line of $\mathcal{G}_k$ and $\mathcal{G}_k^*$ is incident with a thin point, so  there are infinitely many finite partial Sherk planes in which every line is incident with a thin point. Moreover, $\mathcal{G}_k$ has the property that some line is incident only with thin points. On the other hand, if every line of $\mathcal{G}$ is incident with a thick point (which is the case for a Bruck net of degree $r$ where $r$ is even and  $r>2$), then every line of $\mathcal{G}_k^*$ is incident with a thick point. In fact, if $\mathcal{G}$ is an $(n,r)$-net where $r$ is even and $r>2$, then for $k=n(r-1)$,  $\mathcal{G}_{k}$  has the property that every line is incident with a constant number of points, namely $nr=n+k$; similarly, if $n$ is odd and $k=n(r-1)+1$, then $\mathcal{G}_k^*$  has the property that every line is incident with a constant number of points, namely $nr+1=n+k$.

\bibliographystyle{line}
\bibliography{references}

\end{document}